\documentclass[12pt, reqno]{amsart}
 
\usepackage{amsmath,amsthm,amscd,amsfonts,amssymb,graphicx,color}
\usepackage[bookmarksnumbered,colorlinks,plainpages]{hyperref}

\hypersetup{colorlinks=true,linkcolor=black, anchorcolor=black,
	citecolor=black, urlcolor=black, pdftoolbar=true}
\newtheorem{theorem}{Theorem}
\newtheorem{lemma}{Lemma}

\newtheorem{corollary}{Corollary}
\theoremstyle{definition}
\newtheorem{definition}{Definition}
\newtheorem{example}{Example}

\theoremstyle{remark}
\newtheorem{remark}{Remark}
\numberwithin{equation}{section}

\begin{document}
\setcounter{page}{1}
\begin{center}
{\textbf{\Large Non-linear programming problem for semi strongly $E$-preinvexity}}\\
\bigskip
{\textbf {Akhlad Iqbal and Askar Hussain}}\\
{\text{akhlad6star@gmail.com and askarhussain59@gmail.com}}\\Department of Mathematics\\
 Aligarh Muslim University, Aligarh-202002, India\\[0pt]
  \end{center}
 \bigskip
  \noindent
\textbf{Abstract}: 
In this article, we present semi strongly $E$-preinvexity and semi strongly $E$-invexity. To demonstrate the existence of these functions, certain nontrivial examples have been developed. Several significant relationships and characterizations of these functions on strongly $E$-invex sets are discussed. Furthermore, we consider a non-linear programming problem for semi strongly $E$-preinvex functions and investigate relationships between the set of optimal solutions and these functions.\\



\section{\bf{Introduction}}

\noindent
 Convexity and its generalizations is an important branch of mathematical analysis, which have many applications in pure and applied mathematics such as optimization theory, science and engineering see: \cite{Bazaara,Akhlad12,Akhlad11,Jeyakumar,Pini,Suneja,youness2}. Fulga et al. \cite{Fulga} introduced $E$-preinvexity with the support of nontrivial examples, considered a non-linear programming problem for our results, and also showed that a strict local minimum point is a strict global minimum point for a non-linear programming problem.
  Youness \cite{Youness3} introduced the class of strongly $E$-convex function which has very useful properties and results related to this class of functions and satisfies certain local-global minimum points criteria for non-linear programming problem. The extension of strongly $E$-convexity to geodesic strongly $E$-convexity from linear space to Riemannian manifolds via geodesics and techniques has been developed by Kilickman et al. \cite{Adem1}.  An attempt has been made to preserve several interesting properties and results of strongly $E$-convexity for geodesic strongly $E$-convexity, several interesting properties of their geodesic strongly $E$-convexity discussed in \cite{Adem1}.\\
 
 Motivated by Youness \cite{Youness3} and  Kilickman \cite{Adem1}, Hussain \cite{Hussain1} introduced the quasi strongly $E$-convexity which is extension of strongly $E$-convexity. Several interesting properties and results to this function have been discussed. Later, the concept of strongly $E$-preinvexity was studied by Iqbal et al., \cite{Hussain2, Hussain3} which generalized the concept of $E$-preinvexity defined by Fulga, \cite{Fulga}. Further, Youness \cite{Youness4} generalized the concept of strongly $E$-convexity to semi strongly $E$-convexity and considered a non-linear programming problem for semi strongly $E$-convex function. Attempts have been made to preserve several properties and results of semi strongly $E$-convexity and strongly $E$-preinvexity for this new class of semi strongly $E$-preinvexity.
 
\noindent
Motivated and inspired by research works see: \cite{Tair,Ben,Adem2,Yang,Youness1}, we introduce the class of semi strongly $E$-preinvexity and semi strongly $E$-invexity. This paper is divided as follows: Section \ref{2P} contains preliminaries and Section \ref{3S} contains the several interesting  properties  and definitions. Some nontrivial examples have been constructed in support of these definitions. An important relationship and characterization of these functions have been established. In Section \ref{4S},  a non-linear programming problem for semi strongly E-preinvex functions has been considered. We discuss the uniqueness of a global optimal solution and present a condition for a set of optimal solutions to be strongly $E$-invex. We conclude this article in Section \ref{5conl}.

\section{\bf{Preliminaries}}\label{2P}
   The concept of $E$-convexity with respect to map $E:R^{n}\rightarrow R^{n}$ was defined by Youness \cite{Youness1}. Later, they \cite{Youness3,Youness4} extended this concept to the strongly $E$-convexity and semi strongly $E$-convexity as follows:  
   \begin{definition}
   	\cite{Youness3}\label{1D}
   	A non empty set $S\subseteq R^{n}$ is said to be a strongly $E$-convex (SEC) set with respect to (w.r.t.) a map $E:R^{n}\rightarrow R^{n}$, if  $\forall s,t\in S$,~$\alpha\in[0,1]~\lambda\in[0,1]$, we have
   	\begin{eqnarray*}
   		\lambda\left(\alpha s+Es\right)+(1-\lambda)\left(\alpha t+Et\right)\in S.
   	\end{eqnarray*} 
   
   \end{definition}
   \begin{definition}\label{2D}
   	\cite{Youness4}
   	Let $S\subseteq R^{n}$ be a SEC set. A real valued function $h:S\subseteq R^{n}\rightarrow R$ is said to be semi strongly $E$-convex (SSEC) w.r.t. a map  $E:R^{n}\rightarrow R^{n}$ on $S$, if
   	\begin{eqnarray*}
   		h(\lambda(\alpha s+Es)+(1-\lambda)(\alpha t+Et))\leq \lambda h(s)+(1-\lambda)h(t),
   	\end{eqnarray*} 
   $\forall s,t\in S,~\alpha\in[0,1]~\lambda\in[0,1]$.\\
   
   \noindent
   	If the above inequality is strict for every $s,t\in S$,~ $\alpha s+Es\neq \alpha t+Et,~\alpha\in[0,1],\lambda\in(0,1)$, then $h$ is said to be strictly semi strongly $E$-convex (SSSEC) function. 
   \end{definition}
\begin{definition}\label{2D}
	\cite{Youness4}
	Let $S\subseteq R^{n}$ be a SEC set. A real valued function $h:S\subseteq R^{n}\rightarrow R$ is said to be semi quasi strongly $E$-convex (SQSEC) w.r.t. a map  $E:R^{n}\rightarrow R^{n}$ on $S$, if
	\begin{eqnarray*}
		h(\lambda(\alpha s+Es)+(1-\lambda)(\alpha t+Et))\leq\max \{ h(s),h(t)\},
	\end{eqnarray*} 
	$\forall s,t\in S,~\alpha\in[0,1]~\lambda\in[0,1]$.\\
	
	\noindent
	If the above inequality is strict for every $s,t\in S$,~ $\alpha s+Es\neq \alpha t+Et,~\alpha\in[0,1],\lambda\in(0,1)$, then $h$ is said to be strictly semi quasi strongly $E$-convex (SSQSEC) function. 
\end{definition}
    Iqbal et al. \cite{Hussain2} defined strongly $E$-preinvexity as follows:
  \begin{definition}\cite{Hussain2}\label{Defn 1}
  	A set $S\subseteq R^{n}$ is said to be strongly $E$-invex (SEI)  w.r.t. $\Psi\colon R^{n}\times R^{n}\to R^{n}$, if  
 \begin{eqnarray*}
 	\alpha t+Et+\lambda\Psi(\alpha s+Es,\alpha t+Et)\in S,
 \end{eqnarray*}
  $\forall s,t\in S,~\alpha\in[0,1]~\lambda\in[0,1]$.\\
  
  \noindent
 	If $\Psi(\alpha s+Es, \alpha t+Et)=\alpha (s-t)+(Es-Et)~\forall s,t\in R^{n}$, then Definition \ref*{Defn 1} reduces to Definition \ref{1D} defined by
 	Youness \cite{Youness3}.
 \end{definition}
  \begin{definition}\cite{Hussain2} 
 	   A function $h:S\rightarrow R$ is said to be strongly $E$-preinvex (SEP) w.r.t. $\Psi\colon R^{n}\times R^{n}\to R^{n}$ on SEI set $S$, if 
 \begin{eqnarray*}
 	 h(\alpha t+Et+\lambda\Psi(\alpha s+Es,\alpha t+Et))\leq \lambda h(Es)+(1-\lambda)h(Et),
 \end{eqnarray*}
$\forall s,t\in S$, $\alpha\in[0,1],\lambda\in[0,1]$.
 \end{definition}
 \begin{definition} \cite{Hussain2}
 	  A function $h:S\subseteq R^{n}\rightarrow R$ is said to be pseudo strongly $E$-preinvex (PSEP) w.r.t. $\Psi$ on SEI set $S$, if $\exists$ a strictly positive function $b:R^{n}\times R^{n}\rightarrow R$ s.t. $h(Es)< h(Et)$  
 	\begin{eqnarray*}\implies h(\alpha t+Et+\lambda\Psi(\alpha s+Es,\alpha t+Et))\leq h(E(t))+\lambda(\lambda-1)b(Es,Et),
 	\end{eqnarray*} $\forall~s,t \in S$, $\alpha \in[0,1],~\&~\lambda \in[0,1]$. 
 \end{definition}
 \begin{definition}\cite{Hussain3}
  	A function $h:R^{n}\rightarrow R$ is said to be quasi strongly $E$-preinvex (QSEP) w.r.t. $\Psi$ on SEI set $S$, if  
  	\begin{eqnarray*}
  		h(\alpha t+Et+\lambda\Psi(\alpha s+Es,\alpha t+Et))\leq \max\{h(Es),h(Et)\},
  		\end{eqnarray*}
  	 $\forall~s,t \in S$, $\alpha \in[0,1],~\&~\lambda \in[0,1]$.\\
  	 
    \noindent
  	If the above inequality is strict and $h(Es)\neq h(Et),~\forall~s,t\in S~\lambda\in(0,1)$, then $h$ is said to be strictly quasi strongly $E$-preinvex(SQSEP) function.
  \end{definition}

\begin{definition}\cite{Hussain2}\label{5P} 
	Let ${S}\subseteq  {R}^{n}$ be an open SEI set and  ${h}\colon {S}\rightarrow {R}$ be a differentiable function on ${S}$. Then, ${h}$ is said to be strongly $E$-invex (SEI) w.r.t. $ {\Psi}$ on $ {S}$, if
	\begin{eqnarray*}
		\nabla  {h}( {E}t)\ {\Psi}( {\alpha} s+ {E}s, {\alpha} t+ {E}t)^{T}\leq  {h}( {E}s)- {h}( {E}t),
	\end{eqnarray*}
	$\forall s,t\in  {S}$ and $ {\alpha}\in {[0,1]} $.
\end{definition}

   Iqbal et al. \cite{Hussain2} introduced the $Condition~{A}$ as follows: \\
 
     \noindent 
     $\mathbf{Condition}~\mathbf{A}$.\label{Condition A} Let $ E: S\rightarrow  S$ be an onto map and $S\subseteq R^{n}$ be a SEI set w.r.t. $\Psi$. Then, for every $s,t\in S, \alpha\in [0,1],~ \lambda\in [0,1], \exists ~~\bar{t}\in S$ s.t.  
     \begin{eqnarray*} E\bar{t}= \alpha t+Et+\lambda\Psi(\alpha s+Es,\alpha t+Et)\in  S.
     	\end{eqnarray*}
     $\Psi$ satisfies $Condition~ A$ if:  
     \begin{equation*}
     	{\bf A_{1}:~~}~~ {\Psi}\left({\alpha} t+ {E}t, {\alpha} \bar{t}+ {E}\bar{t})\right)=- {\lambda}\left( {\alpha}\bar{t}+ {\Psi}({\alpha} s+ {E}s, {\alpha} t+{E}t)\right), 
     \end{equation*}\begin{equation*}
     	\hspace{.5cm} {\bf A_{2}:~~}~~{\Psi}\left({\alpha}s+ {E}s,{\alpha} \bar{t}+{E}\bar{t})\right)=(1-{\lambda})\left({\alpha}\bar{t}+{\Psi}({\alpha}s+{E}s,{\alpha}t+{E}t)\right). 
     \end{equation*}
 
 \noindent
     For $ \alpha=0$ and $E$ is identity map,~$Condition~{A}$ reduces to the $Condition~{C}$ defined by Mohan et al. \cite{Neogy}. 
     
     \begin{theorem} \cite{Hussain2}\label{1T}
     	Let ${S}\subseteq  {R}^{n}$ be an open SEI set and  ${E}:{S}\rightarrow {S}$ be an onto map. If ${h}: {S}\rightarrow {R}$ is differentiable SEI function w.r.t. ${\Psi}$ on $ {S}$ and ${\Psi}$ satisfies {Condition}~{A}. Then, ${h}$ is SEP w.r.t. $ {\Psi}$ on ${S}$.
     \end{theorem}
    \section{\bf{Semi strongly $E$-preinvex (semi quasi strongly $E$-preinvex) functions }}\label{3S}
        Now, by the importance of generalized convexity, we define semi strongly $E$-preinvex and semi quasi strongly $E$-preinvex functions as follows: Throughout the paper $\Psi\colon R^{n}\times R^{n}\to R^{n}$ is a map.
         
         \begin{definition} \label{5} 
         	  A function $h:S\rightarrow R$ is said to be semi strongly $E$-preinvex (SSEP) w.r.t. $\Psi$ on SEI set $S$, if 
         	\begin{eqnarray*}h(\alpha t+Et+\lambda\Psi(\alpha s+Es,\alpha t+Et))\leq \lambda h(s)+(1-\lambda)h(t),
         	\end{eqnarray*} 
         $\forall~s,t \in S$, $\alpha \in[0,1],~\&~\lambda \in[0,1]$.\\
         
         \noindent 
         	If the inequality is strict and $\alpha s+Es\neq \alpha t+Et,~ \forall ~ s,t\in S$,~~$\alpha\in[0,1]$ and $\lambda\in(0,1)$,   $h$ is said to be strictly semi strongly $E$-preinvex (SSSEP).
         	 For $\alpha=0$, it is semi $E$-preinvex function defined by Syau \cite{Ru Yu}.
         	  \end{definition}
         	  
         	  \begin{example}
         	  	Let  $h:R\rightarrow R$ be defined as 
         	  	\begin{equation*}
         	  	h(s)=
         	  	\left\{\begin{array}{ll}
         	  	\sqrt{s},&\text{if}~s> 0,
         	  	
         	  	\\
         	  	
         	  	0,&\text{if}~s\leq0,

         	  	\end{array}\right.
         	  	\end{equation*}  
         	  	$E:R\rightarrow R$ be defined as 	$Es=0,~~\forall~ s\in R$ and $\Psi:R\times R\rightarrow R$ be  defined as \begin{equation*}
         	  	\Psi(s,t)=
         	  	\left\{\begin{array}{ll}
         	  	-t,&\text{if}~s\neq t, 
         	  	\\
         	  	
         	  	0,&\text{if}~s=t. 
         	  	\end{array}\right.\hspace*{.5cm} 
         	  	\end{equation*} 
         	  	The function $h$ is SSEP and semi $E$-preinvex  w.r.t. $\Psi$ but not SSEC. Particularly, at the points $s=1,t=2,\alpha=1~\&~\lambda=\dfrac{1}{2},$
         	  	\begin{eqnarray*}h(\lambda(\alpha s+Es)+(1-\lambda)(\alpha t+Et))=1.224
         	  		\end{eqnarray*} 
         	  	however,~~~
         	  	\hspace*{4.3cm} $\lambda h(s)+(1-\lambda)h(t) =1.207,$\\ which implies that  
         	  	\begin{eqnarray*} h(\lambda(\alpha s+Es)+(1-\lambda)(\alpha t+Et))\nleq \lambda h(s)+(1-\lambda)h(t).
         	  	\end{eqnarray*}
         	  \end{example}
         	  
         	  \noindent
         	  A semi $E$-preinvex function $h$ w.r.t. $\Psi$  need not be SSEP, as shown by the following example:
         	  \begin{example}
         	  		Let $h:R\rightarrow R$ be defined as
         	  
         	  		\begin{equation*}
         	  	\hspace*{1.2cm} 	h(s)=
         	  		\left\{\begin{array}{ll}
         	  		7, &\text{if}~s<1,~~ or~~ s>4,\\
         	  		
         	  		s-3,&\text{if}~1\leq s<2,\\
         	  		3-s, &\text{if}~2\leq s\leq 3,\\
         	  		
         	  		s-3,&\text{if}~ 3<s\leq 4,\\
         	  		\end{array}\right.
         	  		\end{equation*}  
         	  		$E:R\rightarrow R$ be defined as 	\begin{equation*}
         	  		\hspace*{2.4cm} E(s)=
         	  		\left\{\begin{array}{ll}
         	  		1,&\text{if}~1\leq s\leq 4,\\
         	  		1+\dfrac{2}{\pi}\arctan(1-s),&\text{if}~s<1,\\
         	  		
         	  		2+\dfrac{4}{\pi}\arctan(s-4),&\text{if}~s>4,\\
         	  		
         	  		\end{array}\right.
         	  		\end{equation*}   and  $\Psi:R\times R\rightarrow R$ be defined as
         	  		\begin{equation*} 
         	  		\Psi(s,t)=
         	  		\left\{\begin{array}{ll}
         	  		-t,&\text{if}~s\neq t, 
         	  		~~~~\\
         	  		
         	  		0, &\text{if}~ s=t. 
         	  		\end{array}\right.\hspace*{1.6cm}  
         	  		\end{equation*}
         	  			The function $h$ is a semi $E$-preinvex w.r.t. $\Psi$, but not SSEP. Particularly, at the points $s=1,t=4,\alpha=1$ and $\lambda=1$
         	  			\begin{eqnarray*}
         	  			h(\alpha t+Et+\lambda\Psi(\alpha s+Es,\alpha t+Et))&=&h(0)\\&=&7	 
         	  			\end{eqnarray*}
         	  			 
         	  			However,
         	  				 \begin{eqnarray*}
         	  				 	\hspace{3cm}\lambda h(s)+(1-\lambda)h(t)&=&h(1)\\&=&-2,
         	  				 \end{eqnarray*} 
         	  			\begin{eqnarray*}\implies h(\alpha t+Et+\lambda\Psi(\alpha s+Es,\alpha t+Et))\nleq \lambda h(s)+(1-\lambda)h(t).
         	  				\end{eqnarray*}
         	  		Hence,  $h$ is not SSEP.
         	  
         	  \end{example}
         	   \begin{definition}\label{6}
         	     A function $h:R^{n}\rightarrow R$ is said to be semi quasi strongly $E$-preinvex (SQSEP) w.r.t. $\Psi$ on SEI set $S$, if 
         	   	\begin{eqnarray*} h(\alpha t+Et+\lambda\Psi(\alpha s+Es,\alpha t+Et))\leq \max\{ h(s),h(t)\},
         	   		\end{eqnarray*}
            		$\forall s,t\in S,\alpha\in[0,1],~\lambda\in[0,1]$.\\  
         	   	If the inequality is strict and $h(s)\neq h(t)$, for every $s,t\in S$,~~$\forall~ \alpha\in[0,1]$ and $\lambda\in(0,1)$, then $h$ is said to be strictly semi quasi strongly $E$-preinvex (SSQSEP).
         	   	If $\alpha=0$, then it is semi $E$-prequasiinvex function   defined by Syau et al.\cite{Ru Yu}. 
         	   \end{definition}
         	    
         	   \begin{example}
         	   	Let $h:R\rightarrow R$ be defined as 
         	   	\begin{equation*}
         	   	h(s)=
         	   	\left\{\begin{array}{ll}
         	   	
         	   	1,&\text{if}~s>0, 
         	   	~~~~\\
         	   	
         	   	-s,&\text{if}~s\leq0,
         	   	\end{array}\right.\hspace*{0.2cm}  
         	   	\end{equation*} 
         	   	
         	   	\noindent
         	   	$E:R\rightarrow R$ be a map defined as 	$Es=0,~~~~\forall ~ s\in R$ and  $\Psi:R\times R\rightarrow R$ be defined as \begin{equation*}
         	   	\Psi(s,t)=
         	   	\left\{\begin{array}{ll}
         	   	-t,&\text{if}~s\neq t,\\
         	    
         	   	0,&\text{if}~s=t,\\
         	   	\end{array}\right.\hspace*{0.7cm}  
         	   	\end{equation*} 
         	   	The function $h$ is SQSEP and semi $E$-prequasiinvex  w.r.t. $\Psi$ but not QSEP. Particularly, at the points $u=0,v=1,\alpha=\dfrac{1}{2}$ and $\lambda=\dfrac{1}{2}$
         	   	\begin{eqnarray*}
         	   	h(\alpha t+Et+\lambda\Psi(\alpha s+Es,\alpha t+Et))&=&h\left(\dfrac{1}{4}\right)\\ &=&1	 
         	   	\end{eqnarray*}
         	    
         	   	However,
         	   		\hspace*{4.4cm}$\max\{h(Es),h(Et)\} =0,$  
         	   	\begin{eqnarray*}\implies h(\alpha t+Et+\lambda\Psi(\alpha s+Es,\alpha t+Et))\nleq \max\{h(Es),h(Et)\}.
         	   		\end{eqnarray*}
         	    
         	   	\end{example}
         	   
         	   \noindent
         	   A semi $E$-prequasiinvex w.r.t. $\Psi$  need not be SQSEP as shown in the following example:
         	   \begin{example}
         	   	Let $h:R\rightarrow R$ be a function defined as 
         	   	\begin{equation*}
         	   \hspace*{1cm} 	h(s)=
         	   	\left\{\begin{array}{ll}
         	   	7,&\text{if}~s<1,~~ or~~ s>4,\\
         	   	
         	   	s-3,&\text{if}~1\leq s<2,\\
         	   	3-s,&\text{if}~2\leq s\leq 3,\\
         	   	
         	   	s-3,&\text{if}~3<s\leq 4,\\
         	   	\end{array}\right.
         	   	\end{equation*}  
         	   	$E:R\rightarrow R$ be a map defined as 	
         	   	\begin{equation*}
         	   	\hspace*{2.3cm} E(s)=
         	   	\left\{\begin{array}{ll}
         	   	1,&\text{if}~1\leq s\leq 4,\\
         	   	1+\dfrac{2}{\pi}\arctan(1-s),&\text{if}~s<1,\\
         	   	
         	   	2+\dfrac{4}{\pi}\arctan(s-4),&\text{if}~s>4,
         	   	
         	   	\end{array}\right.
         	   	\end{equation*}   and  $\Psi:R\times R\rightarrow R$ be defined as
         	   	\begin{equation*} 
         	   	\Psi(s,t)=
         	   	\left\{\begin{array}{ll}
         	   	-t,&\text{if}~s\neq t,\\
         	   	~~~~\\
         	   	
         	   	0,&\text{if}~s=t.\\
         	   	\end{array}\right.\hspace*{1.8cm}  
         	   	\end{equation*}
         	   	The function $h$ is semi $E$-prequasiinvex w.r.t. $\Psi$, but not  SQSEP. Particularly, at the points $s=1,t=4,\alpha=1~\&~\lambda=1$
         	   	\begin{eqnarray*} h(\alpha t+Et+\lambda\Psi(\alpha s+Es,\alpha t+Et))=h(0)=7
         	   		\end{eqnarray*}
         	   	However,
         	   	\hspace*{6.1cm} $\max\{ h(s),h(t)\}=1,$\\  
         	   	\hspace*{0.5cm}$\implies h(\alpha t+Et+\lambda\Psi(\alpha s+Es,\alpha t+Et))\nleq \max\{h(s),h(t)\}.$
         	   	 
         	   \end{example}
                 \begin{theorem}\label{1a} 
         	  If $h:R^{n}\rightarrow R$ is SSEP (SQSEP) function w.r.t. $\Psi$ on SEI set $S$, then $ h(\alpha t+Et)\leq h(t)$, $\forall~t\in S,~ \alpha \in[0,1]$.  
         \end{theorem}
         \begin{proof}
         	Recall that $h$ is SSEP  on a SEI set $S$, then $\forall~s,t\in S,\alpha \in[0,1]$ and $\lambda\in[0,1]$, we get
         	\begin{equation*}
         		\alpha t+Et+\lambda\Psi(\alpha s+Es,\alpha t+Et)\in S
         	\end{equation*} and
         	\begin{equation*}
         		h(\alpha t+Et+\lambda\Psi(\alpha s+Es,\alpha t+Et))\leq \lambda h(s)+(1-\lambda)h(t).
         	\end{equation*}
         	Thus, for $\lambda=0$, we have $ h(\alpha t+Et)\leq h(t),~ \forall~t\in S~\&~ \alpha \in[0,1].$   
         \end{proof}
         \begin{theorem}\label{2a} 
         	  If $h_{i}:R^{n}\rightarrow R$, $1\leq i\leq k$, are all SSEP (SQSEP) functions w.r.t. $\Psi$ on SEI set $S$, then the linear combination of SSEP (SQSEP) is also SSEP (SQSEP), $i.e.,$ for $a_{i}\geq0, 1\leq i\leq k$,  $h=\sum\limits_{i=1}^{k}a_{i} h_{i}$ is SSEP (SQSEP) on $S$.  
         \end{theorem}
         \begin{proof}
         		Recall that $h_{i}, 1\leq i\leq k$, are semi strongly $E$- preinvex   on $S$, then $\forall~s,t\in S$, $\alpha\in[0,1]~\&~\lambda\in[0,1]$, we get
         	\begin{eqnarray*}\alpha t+Et+\lambda\Psi(\alpha s+Es,\alpha t+Et)\in S
         		\end{eqnarray*} and
         	 \begin{eqnarray*}
         	 	h(\alpha t+Et+\lambda\Psi(\alpha s+Es,\alpha t+Et))&=&\sum\limits_{i=1}^{k}a_{i} h_{i}(\alpha t+Et+\lambda\Psi(\alpha s+Es,\alpha t+Et))\\
         		&\leq & \lambda\sum\limits_{i=1}^{k}a_{i} h_{i}(s)+(1-\lambda)  \sum\limits_{i=1}^{k}a_{i} h_{i}(t)\\
         		&=& \lambda h(s)+(1-\lambda)h(t).
         	\end{eqnarray*}
         	Hence, $h(s)$ is semi strongly $E$- preinvex on $S$.
         \end{proof}
         
          \begin{theorem}\label{3a} 
          	Let $\{h_{i}\}_{i\in I}$ be a collection of functions defined on SEI set $S$ s.t. $\sup\limits_{i\in I}h_{i}(u)$ exists in $R$, $\forall~ u\in S$. Let $h:S\rightarrow R$ be a function defined by $h(s)=\sup\limits_{i\in I}h_{i}(s),\forall~s\in S$. If the functions $h_{i}:S\rightarrow R$ , $\forall~i\in I$, are all SSEP (SQSEP) w.r.t. $\Psi$ on $S$, then $h$ is SSEP (SQSEP) on $S$.   
          \end{theorem}
          \begin{proof}
          	Assume that the functions $h_{i}:S\rightarrow R$, $\forall i\in I$, are SSEP on $S$. Then, For every $s,t\in S$, $\alpha \in [0,1]$ and $\lambda \in[0,1]$, we get
          	\begin{eqnarray*}
          		h_{i}(\alpha t+Et+\lambda\Psi(\alpha s+Es,\alpha t+Et))&\leq& \lambda h_{i}(s)+(1-\lambda)h_{i}(t),\\
          		\sup\limits_{i\in I}h_{i}(\alpha t+Et+\lambda\Psi(\alpha s+Es,\alpha t+Et))&\leq& \sup\limits_{i\in I} \lambda h_{i}(s)+\sup\limits_{i\in I}(1-\lambda)h_{i}(t)\\
          		&=& \lambda \sup\limits_{i\in I}h_{i}(s)+(1-\lambda) \sup\limits_{i\in I}h_{i}(t)\\
          		h(\alpha t+Et+\lambda\Psi(\alpha s+Es,\alpha t+Et))&\leq& \lambda h(s)+(1-\lambda)h(t).
          	\end{eqnarray*} 
          	 
          \end{proof}
          \begin{theorem}\label{4a}
          	Let the function $h:R^{n}\rightarrow R$ be a SSEP (SQSEP) w.r.t. $\Psi$ on SEI set $S$. Let $g:R\rightarrow R$ be a positively homogeneous increasing function. Then, $g\circ h$ is SSEP (SQSEP) on $S$.
          \end{theorem}
          \begin{proof}
          	Suppose the function $h$ is SSEP on SEI set $S$.  For every $s,t\in S$, $\alpha\in[0,1]$ and $\lambda\in[0,1]$, we have\\
          \begin{eqnarray*}\alpha t+Et+\lambda\Psi(\alpha s+Es,\alpha t+Et)\in S,
          	\end{eqnarray*}  
          	and
          	\begin{eqnarray*}h(\alpha t+Et+\lambda\Psi(\alpha s+Es,\alpha t+Et))\leq\lambda h(s)+(1-\lambda)h(t).
          		\end{eqnarray*}
          	Since $g$ is a  positively homogeneous increasing function, we get
          	\begin{eqnarray*}
          		g\circ h(\alpha t+Et+\lambda\Psi(\alpha s+Es,\alpha t+Et))&\leq& g\circ(\lambda h(s)+(1-\lambda)h(t)\\
          		&=& \lambda(g\circ h)(s)+(1-\lambda)(g\circ h)(t).
          	\end{eqnarray*}
           \end{proof}

          We establish an important characterization of SEP and SSEP in the following theorem.
          \begin{theorem}\label{6TT}
          	  If $h:R^{n}\rightarrow R$ is SEP function w.r.t. $\Psi$ on SEI set $S$, then $h$ is SSEP on $S$ if and only if $h(Et)\leq h(t)$,~$\forall~t\in S$. 
          \end{theorem}
          \begin{proof}
          		Recall that $h$ is strongly $E$-preinvex function on $S$, then  $\forall~s,t\in S,\alpha\in[0,1]$ and $\lambda\in[0,1]$, we get 
          	\begin{eqnarray*}h(\alpha t+Et+\lambda\Psi(\alpha s+Es,\alpha t+Et))\leq \lambda h(Es)+(1-\lambda)h(Et).
          		\end{eqnarray*}
          	Since $h(Et)\leq h(t)$, for each $t\in S$, we have 
          	\begin{eqnarray*}h(\alpha t+Et+\lambda\Psi(\alpha s+Es,\alpha t+Et))\leq \lambda h(s)+(1-\lambda)h(t).
          	\end{eqnarray*}
          	Hence, $h$ is SSEP on $S$.\\
          	Converse part can be proved by taking $\alpha=0$ in the Theorem \ref{1a}. 
          \end{proof}
      \begin{theorem}\label{6T}
       If $h:R^{n}\rightarrow R$ is SSEP function w.r.t. $\Psi$ on SEI set $S$ and $h(s)\leq h(t),\forall s,t\in S$, then $h$ is SQSEP on $S$.  
      \end{theorem}
      \begin{proof}
      	By the Definition \ref{5}, the proof is obvious.
      \end{proof}
     \begin{theorem}
   If $h:R^{n}\rightarrow R$ is quasi strongly $E$-preinvex function w.r.t. $\Psi$ on SEI set $S$ and $h(Et)\leq h(t),\forall~ t\in S$, then $h$ is SQSEP on $S$.  
  \end{theorem}
  \begin{proof}
  	By the Definition \ref{6}, the proof is  obvious.
  \end{proof}
      Now, we define semi pseudo strongly $E$-preinvex function on SEI set as follows:
      \begin{definition}
             A function $h:R^{n}\rightarrow R$ is said to be semi pseudo strongly $E$-preinvex (SPSEP) w.r.t. $\Psi$ on SEI set $S$, if $\exists$ a strictly positive function $b:R^{n}\times R^{n}\rightarrow R$ s.t. $h(s)< h(t)~\implies$ 
             \begin{eqnarray*}h(\alpha t+Et+\lambda\Psi(\alpha s+Es,\alpha t+Et))\leq h(t)+\lambda(\lambda-1)b(s,t),
             	\end{eqnarray*} $\forall~ s,t \in S,~\alpha\in[0,1]~\&~\lambda \in[0,1].$
           \end{definition}
           In the next theorem, we discuss an important relationship between SSEP and SPSEP.
           \begin{theorem}\label{8T} 
           	Let the function $h:S\subseteq R^{n}\rightarrow R$ be SSEP on SEI set $S$. Then, $h$ is semi pseudo strongly $E$-preinvex on $S$.
           \end{theorem}
           \begin{proof}
           	Let $h(s)<h(t)$. Recall that $h$ is semi strongly $E$-preinvex on SEI set $S$. For any $s,t\in S, ~\alpha\in[0,1]~\&~\lambda\in[0,1]$, we get
           	\begin{eqnarray*}
           		h(\alpha t+Et+\lambda\Psi(\alpha s+Es,\alpha t+Et))&\leq& \lambda h(s)+(1-\lambda)h(t)\\ 
           		&=&  h(t)+\lambda(h(s)-h(t))\\
           		&\leq& h(t)+\lambda(1-\lambda)(h(s)-h(t))\\
           		&=&h(t)+\lambda(\lambda-1)(h(t)-h(s))\\
           		&=&h(t)+\lambda(\lambda-1)b(s,t),
           	\end{eqnarray*}
           	
           	\noindent
           	where $b(s,t)=h(t)-h(s)>0$. Therefore, the function $h$ is semi pseudo strongly $E$-preinvex on $S$.
           \end{proof} 
       \begin{definition}
       	Let $S\subseteq R^{n}$ be a SEI set and $h:S\rightarrow R$ be a real valued function. At $r\in R$, the lower level set of $h$ is defined as 
       	\begin{eqnarray*}K_{r}=\{s\in S:h(s)\leq r\}.
       	\end{eqnarray*}
       \end{definition}
       We examine an important characterization of a SQSEP function in terms of its lower level set in the following theorems.
        \begin{theorem}
      	If the lower level set $K_{r}$ is SEI, for any $r\in R$, then the function $h:R^{n}\rightarrow R$ is SQSEP w.r.t. $\Psi$ on SEI set $S$.
       	\end{theorem}
       \begin{proof}
       	Assume that the function $h$ is SQSEP on SEI set $S$. For every $s,t\in K_{r}$,~$\alpha\in[0,1]~\&~\lambda\in[0,1],$ it follows that $h(s)\leq r,~h(t)\leq r$, we get
       	\begin{eqnarray*}\alpha t+Et+\lambda\Psi(\alpha s+Es,\alpha t+Et)\in S,
       		\end{eqnarray*}
       	and
       	\begin{eqnarray*}h(\alpha t+Et+\lambda\Psi(\alpha s+Es,\alpha t+Et))\leq \max\{h(s),h(t)\},
       		\end{eqnarray*}
      	which implies that
       	\begin{eqnarray*}\alpha t+Et+\lambda\Psi(\alpha s+Es,\alpha t+Et)\in K_{r}.
       		\end{eqnarray*}
       	Hence, $K_{r}$ is a SEI set.\\
       	Conversely, assume that  for every $r\in R$, the set $K_{r}$ and $S$ are SEI sets. For every $s,t\in S,~\alpha\in[0,1]~\&~\lambda\in[0,1]$, we get 
       	\begin{eqnarray*}\alpha t+Et+\lambda\Psi(\alpha s+Es,\alpha t+Et)\in S.
       		\end{eqnarray*}
       	Let $s,t\in K_{r}$ and $r=\max\{h(s),h(t)\}$.
       	Since $K_{r}$ is SEI set, then 
       	\begin{eqnarray*}\alpha t+Et+\lambda\Psi(\alpha s+Es,\alpha t+Et)\in K_{r},
       		\end{eqnarray*}
       	and
       	\begin{eqnarray*}h(\alpha t+Et+\lambda\Psi(\alpha s+Es,\alpha t+Et))\leq r=\max\{h(s),h(t)\}.
       		\end{eqnarray*}
       	Hence, $h$ is SQSEP on $S$.
       \end{proof}  
   In the following theorem, we show the relationship between SSEP and lower level set.  	
         \begin{theorem}
         	 If a function $h:R^{n}\rightarrow R$ is SSEP  w.r.t. $\Psi$ on a SEI set $S$, then the lower level set $K_{r}$ is SEI, $\forall r\in R$.
         \end{theorem}
         \begin{proof}
         	For any $s,t\in K_{r},~ \alpha\in[0,1]~\&~\lambda\in[0,1]$, it follows that $h(s)\leq r$ and $h(t)\leq r$. Recall that $h$ is semi strongly $E$-preinvex w.r.t. $\Psi$ on SEI set $S$. Therefore, we have
         	\begin{eqnarray*}\alpha t+Et+\lambda\Psi(\alpha s+Es,\alpha t+Et)\in S,
         		\end{eqnarray*} and
         	\begin{eqnarray*}
         		h(\alpha t+Et+\lambda\Psi(\alpha s+Es,\alpha t+Et)) &\leq & \lambda h(s)+(1-\lambda)h(t)
         	\\	&\leq& r.
         	\end{eqnarray*}
         	 
         	$ \implies \alpha t+Et+\lambda\Psi(\alpha s+Es,\alpha t+Et)\in K_{r}.$ Hence, the lower level set $K_{r}$ is SEI.   
         \end{proof}
         \begin{definition}
         	Let $G\subseteq R^{n}\times R$ and $E:R^{n}\rightarrow R^{n}$ be a map. A set $G$ is said to be strongly $G$-invex, if for any $(s,p),~~~(t,q)\in G,~~\alpha\in[0,1]$ and $\lambda\in[0,1]$, we have
         	\begin{eqnarray*}(\alpha t+Et+\lambda\Psi(\alpha s+Es,\alpha t+Et),~\lambda p+(1-\lambda)q)\in G.
         		\end{eqnarray*}
         	 \end{definition}

         	An $epi(h)$ of $h$ is defined as:
         	
         	\begin{eqnarray*}epi(h)=\{(s,\gamma):h(s)\leq\gamma\}.
         	\end{eqnarray*}
         We consider an important characteristic of a SSEP  function in terms of its $epi(h)$ under this theorem.
         \begin{theorem} \label{20TT} 
         	    A function $h:R^{n}\rightarrow R$ is SSEP  w.r.t. $\Psi$ on SEI set $S\iff$ the epi(h) is strongly $G$-invex set on $S\times R$.
         \end{theorem}
         \begin{proof}
          Let $(s,p),~(t,q)\in epi(h)$, it follows that $h(s)\leq p,~h(t)\leq q$~~$,\alpha\in[0,1]$ and $\lambda\in[0,1]$. Recall that $h:R^{n}\rightarrow R$ is SSEP on SEI set $S$. Therefore, we have
          	\begin{eqnarray*}\alpha t+Et+\lambda\Psi(\alpha s+Es,\alpha t+Et)\in S,
          		\end{eqnarray*} and
          \begin{eqnarray*}
           h(\alpha t+Et+\lambda\Psi(\alpha s+Es,\alpha t+Et))&\leq& \lambda h(s)+(1-\lambda)h(t)\\
           &\leq & \lambda p+(1-\lambda)q.
          \end{eqnarray*}
         \begin{eqnarray*}\implies(\alpha t+Et+\lambda\Psi(\alpha s+Es,\alpha t+Et),~\lambda p+(1-\lambda)q)\in epi(h).
         	\end{eqnarray*}
         	Thus, $epi(h)$ is strongly $G$-invex set on $S\times R$.\\
         	Conversely, assume that $epi(h)$ is strongly $G$-invex set. Let $s,t\in S, \alpha\in[0,1]$ and $\lambda\in[0,1]$, then $(s,h(s))\in epi(h)$ and $(t,h(t))\in epi(h)$. Recall that $epi(h)$ is strongly $G$-invex on $S\times R$, we get 
         	\begin{eqnarray*}(\alpha t+Et+\lambda\Psi(\alpha s+Es,\alpha t+Et),~\lambda h(s)+(1-\lambda)h(t))\in epi(h),
         		\end{eqnarray*} 
         	which implies that
         	\begin{eqnarray*}h(\alpha t+Et+\lambda\Psi(\alpha s+Es,\alpha t+Et))\leq \lambda h(s)+(1-\lambda)h(t).
         		\end{eqnarray*}
         	Hence, $h$ is SSEP  on $S$.    
          \end{proof}
          \begin{theorem}
             Suppose that $\{h_{i}\}_{i\in I}$ is a family of SSEP  and bounded above on SEI set $S$. Then,   $h(s)=\sup\limits_{i\in I}h_{i}(s)$ is semi strongly $E$- preinvex on $S$.	 
          \end{theorem}
          \begin{proof}
          	 Recall that $h_{i},~i\in I$, are semi strongly $E$-preinvex on SEI set $S$. Then, the $epi(h_{i})$ of $h_{i}$,
          	 \begin{eqnarray*}epi(h_{i})=\{(s,\gamma):h_{i}(s)\leq\gamma\}
          	 	\end{eqnarray*} is SEI set. Then,
          	 \begin{eqnarray*}
          	 	\bigcap _{i\in I}epi(h_{i})&=&\{(s,\gamma):h_{i}(s)\leq\gamma\}\\
          	 	&=&\{(s,\gamma):h(s)\leq\gamma\}, 
          	 \end{eqnarray*}
          	  
          	  where $h(s)=\sup\limits_{i\in I}h_{i}(s)$, is strongly $G$-invex. Hence, by using Theorem \ref{20TT}, $h$ is semi strongly $E$-preinvex on $S$.  
          \end{proof}
         \begin{theorem}\label{13T} 
         	If the functions $h_{i}:R^{n}\rightarrow R,~ 1\leq i\leq k$, are SSEP (SQSEP) w.r.t. $\Psi$  on $R^{n}$. Then, the set $S=\{s\in R^{n}:h_{i}(s)\leq0,~1\leq i\leq k\}$ is SEI. 
         \end{theorem}
         \begin{proof}
         		Recall that $h_{i}(s),~ 1\leq i\leq k$, are semi strongly $E$-preinvex. For any $s,t\in S\subseteq R^{n},\alpha\in[0,1]~\&~\lambda\in[0,1]$, we have
         	\begin{eqnarray*}
         		h_{i}(\alpha t+Et+\lambda\Psi(\alpha s+Es,\alpha t+Et))&\leq& \lambda h_{i}(s) +(1-\lambda)(h_{i}(t)\\
         		&\leq& 0.
         	\end{eqnarray*}
         Hence,	$\alpha t+Et+\lambda\Psi(\alpha s+Es,\alpha t+Et)\in S.$ Thus, the set $S$ is SEI.
         \end{proof}
         \begin{theorem} \label{14T}
         	  If the functions $h_{i}:R^{n}\rightarrow R, 1\leq i\leq k$ are SSEP (SQSEP) w.r.t. $\Psi$ on $R^{n}$. Then, $S=\bigcap\limits_{i}^{k}\{s\in R^{n}:h_{i}(s)\leq0,~1\leq i\leq k\}$ is SEI.
         \end{theorem}
         \begin{proof}
         	The sets $S_{i}=\{u\in R^{n}:h_{i}(s)\leq 0,~1\leq i\leq k\},~1\leq i\leq k$ are SEI.  Hence, by Theorem \ref{13T}, the intersection~ $S=\bigcap\limits_{i}^{k}S_{i}$~~~ of $S_{i}$ is also SEI set.  
         	\end{proof}

         For the differentiable functions, semi strongly $E$-invexity can be defined as follows:  
         \begin{definition}\label{10D}
         	Let $ S\subseteq R^{n}$ be a SEI set and $h:S\rightarrow R$ be differentiable function on $S$. A function  $h$ is said to be semi strongly $E$-invex (SSEI) w.r.t. $\Psi$ on $S$, if $\forall~s,t\in  S$ and $ \alpha\in [0,1] $, we have
         	\begin{eqnarray*}\nabla h(Et) \Psi (\alpha s+ Es,\alpha t+Et)^{T}\leq  h(s)-h(t).
         		\end{eqnarray*}
         	 If $\alpha=0$, then it reduces to semi $E$-invex function defined by Jaiswal \cite{Jaiswal}.
         \end{definition}
     \begin{example}
     	Consider the map $E:R^{2}\rightarrow R^{2}$ is defined as $E(s,t)=(0,t)$, and the map $\Psi:R^{2}\times R^{2}\rightarrow R^{2}$ is defined as $\Psi((s_{1},t_{1}),(s_{2},t_{2}))=(s_{1}-s_{2},t_{1}-t_{2})$. Then, the set $S=\{(u,v)\in R^{2}:s,t\leq0\}$ is SEI w.r.t. $\Psi$. Let the function $h:S\rightarrow R$ be defined by $h(s,t)=s^{3}+t^{3}$. Then, $h$ is SSEI w.r.t. $\Psi$ on $S$. 
     	For every $s,t\in S,~ s_{1}\leq s_{2},~ t_{1}\leq t_{2}$ and $\alpha\in[0,1]$, we have
     	\begin{eqnarray*}\nabla h(Et)\Psi(\alpha s+Es,\alpha t+Et)^{T}\leq h(s)-h(t).
     		\end{eqnarray*}
     \end{example}
 
 \begin{remark}\label{1r}
 	  By the Theorem \ref{1T}, a differentiable SEI function is SEP function if $\Psi$ satisfies Condition A and by the Theorem  \ref{6TT}, every SEP $h$ is SSEP w.r.t. $\Psi$ on $S$ if and only if $ h(Et)\leq h(t)$,~$\forall~t\in S$. Hence, all the results for SEP functions are true for SSEP function, when $h(Et)\leq h(t)$.
 \end{remark}
    
    The relationship between SSEP  and  SPSEP is discussed in this theorem.
     \begin{theorem} 
     	Let $S\subseteq  R^{n}$ be an open SEI set  w.r.t. $\Psi$ and $E:S\rightarrow S$ be an onto map. If the function $h: S\rightarrow R$ is differentiable SEI w.r.t. $\Psi$ on $S$, $\Psi$ satisfies $Condition~A$ and $h(Et)\leq h(t)$,~$\forall~t\in S$. Then, $h$ is SPSEP  w.r.t. $\Psi$ on $S$.
     \end{theorem}
     \begin{proof}
     	By Remark \ref{1r} and Theorem \ref{8T}, the proof is obvious.
     \end{proof}
         
         On the SEI set, we define the semi quasi strongly $E$-invexity and the semi pseudo strongly $E$-invexity  as follows:
          \begin{definition}   A differentiable function $h:S\rightarrow R$ is said to be semi quasi strongly $E$-invex (SQSEI) w.r.t. $\Psi$ on SEI set $S$, if $\forall~s,t\in S$ and $\alpha\in[0,1]$, we have 
          	\begin{eqnarray*}h(s)\leq h(t) \implies\nabla h(Et)\Psi(\alpha s+Es,\alpha t+Et)^{T}\leq0.
          		\end{eqnarray*}
          	
          	\noindent
          	If $\alpha=0$, then it reduces to semi quasi $E$-invex function defined by Jaiswal \cite{Jaiswal}.
          \end{definition}
          \begin{example}
         Consider $E\colon R^{2}\rightarrow R^{2}$ is defined by $E(s,t)=(0,t)$, and $\Psi\colon R^{2}\times R^{2}\rightarrow R^{2}$ is defined by $\Psi((s_{1},t_{1}),(s_{2},t_{2}))=(s_{1}-s_{2},t_{1}-t_{2})$.	Then, set $S=\{(s,t)\in R^{2}:s,t\leq0\}\subseteq R^{2}$ is SEI w.r.t. $\Psi$. Let the function $h:S\rightarrow R$ be  defined by $h(s,t)=s^{3}+t^{3}$. Then, $h$ is SQSEI w.r.t. $\Psi$ on $S$.  For every $x=(s_{1},t_{1}),y=(s_{2},t_{2})\in S,~s_{1}\leq s_{2},~t_{1}\leq t_{2}$ and $\alpha\in[0,1],$ we have
         	\begin{eqnarray*}h(x)\leq h(y) \implies\nabla h(Ey)\Psi(\alpha x+Ex,\alpha y+Ey)^{T}
         		\end{eqnarray*}
         	\begin{eqnarray*}
         		&=&3t_{2}^{2}(\alpha+1)(t_{1}-t_{2})\\
         		&\leq &0. 
         	\end{eqnarray*}
         	 
         	But $h$ is not SSEI w.r.t. $\Psi$ on $S$. Particularly, at the points $s_{1}=-\dfrac{1}{2},s_{2}=-\dfrac{1}{4},t_{1}=-\dfrac{1}{3},t_{2}=-\dfrac{1}{9}$ and $\alpha=0$,
         	\begin{eqnarray*}
         	0&\leq& h(s)-h(t)-\nabla h(Et)\Psi(\alpha s+Es,\alpha t+Et)^{T}
         	\\&=&s_{1}^{3}-s_{2}^{3}+t_{1}^{3}+2t_{2}^{3}-3t_{2}^{2}t_{1},\end{eqnarray*}
         	which implies  
         	\begin{eqnarray*}\nabla h(Et)\Psi(\alpha s+Es,\alpha t+Et)^{T}\nleq h(s)-h(t).
         		\end{eqnarray*}
         	 \end{example}
          
          \begin{theorem}
            Let $h:S\rightarrow R$ be a differentiable SSEI on SEI set $S$ and $h(s)\leq h(t)$ for every $s,t\in S$. Then, $h$ is SQSEI on $S$.
          \end{theorem}
          \begin{proof}
          	By the Definition \ref{10D}, the proof is obvious.
          \end{proof}

         	  In the following theorem, we explain a relationship between SQSEP and SSEI.
         	 \begin{theorem}
         	 	Let $S\subseteq R^{n}$ be SEI set. Let $h:S\rightarrow R$ be a differentiable SEI function on $S$, satisfies the $Condition~A$ and  $h(Et)\leq h(t)$ for every $t\in S$. Then, $h$ is SQSEP on $S$. 
         	 \end{theorem}
         	  \begin{proof}
         	  	 By Remark \ref{1r} and Theorem \ref{6T} the proof is obvious.
         	  \end{proof}
      \begin{definition}   
      	A differentiable function $h:S\rightarrow R$ is said to be semi pseudo strongly $E$-invex (SPSEI) w.r.t. $\Psi$ on SEI set $S$, if  
           	\begin{eqnarray*}\nabla h(Et)\Psi(\alpha s+Es,\alpha t+Et)^{T}\geq0\implies h(s)\geq h(t),
           	\end{eqnarray*}
           	$\forall~s,t \in S$, $\alpha \in[0,1],~\&~\lambda \in[0,1]$.\\
           	 
           \noindent
           	 If $\alpha=0$, then it reduces to semi pseudo $E$-invex function which was defined by Jabarootian \cite{jabar}.
           \end{definition}
         \begin{example}
         		Consider $E\colon R^{2}\rightarrow R^{2}$ is defined by $E(s,t)=(0,t)$, and $\Psi\colon R^{2}\times R^{2}\rightarrow R^{2}$ is defined by $\Psi((s_{1},t_{1}),(s_{2},t_{2}))=(s_{1}-s_{2},t_{1}-t_{2})$. Then, set $S=\{(s,t)\in R^{2}:s,t\geq0\} \subseteq R^{2}$ is SEI w.r.t. $\Psi$. Let the function $h:S\rightarrow R$ be defined by $h(s,t)=-s^{2}-t^{2}$. Then, $h$ is SPSEI w.r.t. $\Psi$ on $S$.

          	 \noindent
          	But $h$ is not SSEI w.r.t. $\Psi$ on $S$. Particularly, at the points $x=(s_{1},t_{1}),y=(s_{2},t_{2})$, $s_{1}=s_{2}$ and $\alpha=0$,
          	\begin{eqnarray*}
          	\hspace{3.1cm}0&\leq & h(x)-h(y)-\nabla h(Ey)\Psi(\alpha x+Ex,\alpha y+Ey)^{T}\\
          	&=&-(t_{2}-t_{1})^{2}\\ &\leq &0.	 
          	\end{eqnarray*}
           
          	$\implies \nabla h(Et)\Psi(\alpha s+Es,\alpha t+Et)^{T}\nleq h(s)-h(t).$
         	\end{example}
         \begin{theorem}
     	Let $S\subseteq R^{n}$ be SEI set and the function $h:S\rightarrow R$ be SSEI on $S$. If 
     	\begin{eqnarray*}\nabla h(Et)\Psi(\alpha s+Es,\alpha t+Et)^{T}\geq0, ~~\forall ~~s,t\in S,~\alpha\in[0,1].
     	\end{eqnarray*}
     	Then, $h$ is SPSEI on $S$.
        \end{theorem}
       \begin{proof}
     	By the Definition \ref{10D}, the proof is obvious.
     \end{proof}
         
        \section{\bf{Non-linear programming problem}}\label{4S} 
        In this section, we consider the non-linear programming problem for SSEP (SQSEP) functions is formulated as follows:
        \begin{equation} \label{4.1} 
       \hspace*{1cm} (P)~~~ \left\{\begin{array}{ll}
        Min~h(s)\\
       
        h_{j}(s)\leq 0  , ~1\leq j\leq k,\\
       
        s\in R^{n},\\
        
        \end{array}\right.
        \end{equation}
        where $E:R^{n}\rightarrow R^{n}$ is a map and $h:R^{n}\rightarrow R,~h_{j}:R^{n}\rightarrow R,~1\leq j\leq m$, are SSEP (SQSEP) functions on $R^{n}$. Let $X$ be the set of feasible solutions as given below:
        \begin{equation} \label{4.2}
        X=\{s\in R^{n}:~h_{j}(s)\leq 0,~~1\leq j\leq k\}.
        \end{equation}
       
        \begin{lemma}\label{1L}  
           	The set $X$ of feasible solutions of (\ref{4.1}) is a SEI.
           \end{lemma}
           \begin{proof}
             The proof is obvious based on Theorem \ref{13T}.
           	  \end{proof}
       \begin{theorem}\label{36}
       	Let $X\subseteq R^{n}$ be a SEI set and using this inequality $h(\alpha t+Et)\leq h(t),~\forall~ t\in X,\alpha\in[0,1]$. If $t^{*}\in X$ is a solution of the following problem:
       	\begin{eqnarray*}(P_{\alpha})~Min ~h(\alpha t+Et),
       		\end{eqnarray*} 
       	subject to $t\in X$. Then, $\alpha t^{*}+Et^{*}$ represents a solution of (\ref{4.1}).
       	\end{theorem}
       \begin{proof}
       	On contrary suppose $\alpha t^{*}+Et^{*}$ is not a solution of (\ref{4.1}), then  $\exists~t\in X$ s.t.
       	 \begin{eqnarray*} 
       	 	h(t)\leq h(\alpha t^{*}+Et^{*}).
       	 	\end{eqnarray*}
       	Also, 
       	\begin{eqnarray*}
       		h(\alpha t+Et)\leq h(t)\leq h(\alpha t^{*}+Et^{*}),
       		\end{eqnarray*}
       	which contradicts the optimality of $t^{*}$ for the problem $(P_{\alpha})$.\\ Therefore, $\alpha t^{*}+Et^{*}$ is a solution of (\ref{4.1}).  
       \end{proof}
       \begin{theorem}
       	Let $X\subseteq R^{n}$ be a SEI set and $h:R^{n}\rightarrow R$ be a SSEP (SQSEP) w.r.t. $\Psi$ on $X$. If $t^{*}\in X$ is a solution of the following problem:\\
       	\begin{eqnarray*}
       		(P_{\alpha})~~Min ~h(\alpha t+Et),
       		\end{eqnarray*} subject to $t\in X$. Then, $\alpha t^{*}+Et^{*}$ is a solution of (\ref{4.1}).
       \end{theorem}
       \begin{proof}
       	By the Theorem \ref{1a} and Theorem \ref{36}, the proof is obvious.
       \end{proof}
   
   \noindent
    We discuss the uniqueness of global optimal solution from the feasible solution for (\ref{4.1}) as follows:
           \begin{theorem}
           	Let $X\subseteq R^{n}$ be a SEI set. If the function $h:R^{n}\rightarrow R$ is SSSEP w.r.t. $\Psi$ on $X$, then the global optimal solution of (\ref{4.1}) is unique.
           \end{theorem}
           \begin{proof}
           	 Assume that $u,v\in X$ represents two distinct global optimal solutions for (\ref{4.1}). Then, it follows that $h(s)=h(t)=\min\limits_{u\in X}h(u)$.
           	 The set $X$ is SEI and $h:R^{n}\rightarrow R$ is SSSEP by using Lemma \ref{1L}. Therefore,
   
           	  \begin{eqnarray*}
           	  	h(\alpha t+Et+\lambda\Psi(\alpha s+Es,\alpha t+Et))<\lambda h(s)+(1-\lambda)h(t)=h(t),
           	  	\end{eqnarray*} which gives
           	  \begin{eqnarray*}h(\alpha t+Et+\lambda\Psi(\alpha s+Es,\alpha t+Et))< h(t),
           	  \end{eqnarray*}
           	   for any $\alpha\in[0,1],~\lambda\in(0,1)$,
           	 which is a contradiction. Hence, the global optimal solution (\ref{4.1}) is unique.
           \end{proof}
        
       In the following theorem, we discuss a condition for a set of optimal solutions to be SEI.
       \begin{theorem}\label{19}
       	  Let $h:R^{n}\rightarrow R$ be SSEP  function w.r.t. $\Psi$ on SEI set $X\subseteq R^{n}$ and suppose that  $\beta=\min\limits_{s\in X}h(s)$. Then, the set $X_{opt}=\{s\in X:h(s)=\beta\}$ of optimal solutions of (\ref{4.1}) is SEI set. If $h$ is SSSEP  w.r.t. $\Psi$, then the set $X_{opt}$ is singleton.
       \end{theorem}
       \begin{proof}
       	 For every $s,t\in X_{opt}\subseteq X,~\alpha\in[0,1]$ and $\lambda\in[0,1]$. Then $u,v\in X$ and by using Lemma \ref{1L}, we get
       	\begin{eqnarray*}
       		\alpha t+Et+\lambda\Psi(\alpha s+Es,\alpha t+Et)\in X.
       		\end{eqnarray*}
       	 Recall that $h$ is SSEP, then 
       	\begin{eqnarray*}h(\alpha t+Et+\lambda\Psi(\alpha s+Es,\alpha t+Et))=\lambda h(s)+(1-\lambda)h(t)= \beta.
       	\end{eqnarray*}
       	Therefore, 
       	\begin{eqnarray*}\alpha t+Et+\lambda\Psi(\alpha s+Es,\alpha t+Et)\in X_{opt}.
       		\end{eqnarray*}
       	Hence, the set $X_{opt}$ is SEI.\\
       	On the other hand, assume on contrary that $\forall~s,t\in X_{opt},~~s\neq t,~\alpha\in[0,1]$ and $\lambda\in(0,1)$. Then, 
       	\begin{eqnarray*}\alpha t+Et+\lambda\Psi(\alpha s+Es,\alpha t+Et)\in X.
       		\end{eqnarray*}
       	Again, recall that $h$ is SSSEP, we have 
       	\begin{eqnarray*}h(\alpha t+Et+\lambda\Psi(\alpha s+Es,\alpha t+Et))<\lambda h(s)+(1-\lambda)h(t)= \beta,
       		\end{eqnarray*}
       which is a contradiction to $\beta=\min\limits_{s\in X}h(s)$. Thus,  the result has been proved.
       \end{proof}
   Analogous result to Theorem \ref{19} is as follows:
      \begin{theorem} \label{20}
       	Let function $h:R^{n}\rightarrow R$ be SQSEP w.r.t. $\Psi$ on SEI set $X$. Let $\beta=\min\limits_{s\in X}h(s)$. The set $X_{opt}=\{s\in X:h(s)=\beta\}$ of optimal solutions of (\ref{4.1}) is SEI. If  $h$ is SSQSEP w.r.t. $\Psi$ on $X$, then the set $X_{opt}$ is a singleton.
       \end{theorem}
        
       \begin{theorem}
       	If the functions $h:R^{n}\rightarrow R,~h_{j}:R^{n}\rightarrow R,~1\leq j\leq k,$ are SQSEP w.r.t. $\Psi$ on $R^{n}$, then the set of optimal solutions of (\ref{4.1}) is SEI.
       \end{theorem}
       \begin{proof}
       	By Lemma \ref{1L}, then set $X$ is SEI. Hence, by the Theorem \ref{19} the set $X_{opt}=\{s\in X:h(s)=\beta\}$ of optimal solutions of (\ref{4.1}) is SEI.	 
       \end{proof}
       \begin{corollary}
       	If the functions $h:R^{n}\rightarrow R,~h_{j}:R^{n}\rightarrow R,~ 1\leq j\leq k,$ are SSEP w.r.t. $\Psi$ on $R^{n}$, then the set of optimal solutions of (\ref{4.1}) is SEI.
       \end{corollary}
    \begin{theorem}\label{27T}   
           	Let $h:R^{n}\rightarrow R$ be a SSSEP (SSQSEP) function w.r.t. $\Psi$ on $R^{n}$, $h_{j}:R^{n}\rightarrow R,~ 1\leq j\leq m$, be SSEP functions w.r.t. $\Psi$ on
           	$R^{n}$. Suppose $\emptyset\neq X$ represents set of feasible solutions defined by (\ref{4.2}) and $s$ is a fixed point of $E$. If $s$ is a local minimum point of (\ref{4.1}), then $s$ is a strict global minimum point of (\ref{4.1}).
           \end{theorem}
           \begin{proof}
           	Since $s$ is fixed local minimum point of (\ref{4.1}), we get that $s\in R^{n}$, $h_{j}(s)\leq 0,~ 1\leq j\leq m$ and  $~\exists~\epsilon>0$ s.t. $h_{0}(s)\leq h_{0}(t), ~\forall~ t\in B_{\epsilon}(s)\cap X-\{s\},$ where $B_{\epsilon}(s)=\{t\in R^{n}:\|t-s\|<\epsilon\}$ is an open ball. Suppose that $\exists~u\in X,~u\neq s$ s.t. $h_{0}(u)<h_{0}(s).$ Since $u,s\in X$, for any fixed $\alpha\in[0,1]$ and $\lambda\in[0,1]$, due to Lemma \ref{1L}, we get $\alpha s+s^{*}+\lambda\Psi(\alpha u+u^{*},\alpha s+s^{*})\in X.$ For $\alpha u+u^{*}\neq \alpha s+s^{*}$, $\alpha\in[0,1]$ and $\lambda\in[0,1]$, we have
           	\begin{equation}\label{4.3a}
           		h(\alpha s+s^{*}+\lambda\Psi(\alpha u+u^{*},\alpha s+s^{*}))<\lambda h(u)+(1-\lambda)h(s)=h_{0} ~~ (say).
           	\end{equation}
           Four possible cases will arise:
           	 \vspace{1mm}
           	
           	\noindent
           	$\mathbf{Case (1):}$ $\Psi(u^{*},s^{*})=0$ and  $\alpha=0$, then (\ref{4.3a}) results in a contradiction.
           	
           	\vspace{.2cm}
           	\noindent
           	$\mathbf{Case (2):}$ $\Psi(u^{*},s^{*})\neq 0$ and $\alpha=0$, we choose $\bar{\lambda}=\min\left\{1,\dfrac{\epsilon}{\|\Psi(u^{*},s^{*})\|}\right\}$ and, for any $\lambda\in [0,\bar{\lambda})$, we get
           	\begin{eqnarray*}
           		\|s^{*}+\lambda \Psi(u^{*},s^{*})-s^{*}\|&=&\lambda\|\Psi(u^{*},s^{*})\|\\
           		&<&\bar{\lambda}\|\Psi(u^{*},s^{*})\|\\&\leq&\epsilon. 
           	\end{eqnarray*}
           	In this case, we obtained that  $s^{*}+\lambda \Psi(u^{*},s^{*})\in B_{\epsilon}(s)\cap S-\{s\}$,   $\forall\lambda\in(0,\bar{\lambda})$, and from (\ref{4.3a}), we get 
           	\begin{eqnarray*}h(s^{*}+\lambda \Psi(u^{*},s^{*}))<h_{0},
           		\end{eqnarray*}
           	which contradicts the fact that $u$ is a strict global minimum point of (\ref{4.1}).
           	
           	\vspace{.2cm}
           	\noindent
           	$\mathbf{Case (3):}$ $\Psi(\alpha u+u^{*},\alpha s+s^{*})=0$ and $\alpha\neq0$, we choose $\bar{\alpha}=\min\left\{1,\dfrac{\epsilon}{\|s\|}\right\}$ and, for any $\lambda\in [0,\bar{\alpha})$, we get \begin{eqnarray*}
           		 \|\alpha s+s^{*}-s^{*})\|&=&\alpha\| s \|\\
           		 &<&\bar{\alpha}\| s \|\\&\leq&\epsilon.
           	\end{eqnarray*}
           	In this case, we obtained that $\alpha s+s^{*} \in B_{\epsilon}(s)\cap X-\{s\}$,   $\forall\alpha\in(0,\bar{\alpha})$, and from (\ref{4.3a}), we get
           	 \begin{eqnarray*}
           	 	h(\alpha s+s^{*})<h_{0},
           	\end{eqnarray*}
           	which contradicts the fact that $u$ is a strict global minimum point of (\ref{4.1}).
           	
           	\vspace{.2cm}
           	\noindent
           	$\mathbf{Case (4):}$ $\Psi(\alpha u+u^{*},\alpha s+s^{*})\neq 0$ and $\alpha\neq0$, we choose $\bar{\alpha}=\min\left\{1,\dfrac{\epsilon_{1}}{\|s\|}\right\}$ and, for any $\alpha\in [0,\bar{\alpha})$,  we have\\
           	$\|\alpha s+s^{*}+\lambda \Psi\left(\alpha s+u^{*},\alpha s+s^{*}\right)-s^{*}\|$ 
           	\begin{eqnarray*}
           		 &\leq&\alpha\|s\|+\lambda\|\Psi\left(\alpha u+u^{*},\alpha s+s^{*}\right)\|\\
           		 &<&\bar{\alpha}\|s\|+\lambda\|\Psi\left(\bar{\alpha} u+u^{*},\bar{\alpha} s+s^{*}\right)\|\\
           		 &\leq&\epsilon_{1}+\lambda\left|\left|\Psi\left(\dfrac{\epsilon_{1}}{\|s\|} u+u^{*},\dfrac{\epsilon_{1}}{\|s\|} s+s^{*}\right)\right|\right|.
           	\end{eqnarray*}
           	again  we choose $\bar{\lambda}=\min\left\{1,\dfrac{\epsilon_{2}}{\left|\left|\Psi\left(\dfrac{\epsilon_{1}}{\|s\|} u+u^{*},\dfrac{\epsilon_{1}}{\|s\|} s+s^{*}\right)\right|\right|}\right\}$ and, for any $\lambda\in [0,\bar{\lambda})$,
           	\begin{eqnarray*}
           		 &<&\epsilon_{1}+\bar{\lambda}\left|\left|\Psi\left(\dfrac{\epsilon_{1}}{\|s\|} u+u^{*},\dfrac{\epsilon_{1}}{\|s\|} s+s^{*}\right)\right|\right|.\\
           		 &\leq &\epsilon_{1}+\epsilon_{2}\\&=&\epsilon.
           	\end{eqnarray*}
           	In this case, we obtained that  $\alpha s+s^{*}+\lambda\Psi(\alpha u+u^{*},\alpha s+s^{*})\in B_{\epsilon}(s)\cap X-\{s\}$,   $\forall\alpha\in (0,\bar{\alpha})$, $\lambda\in(0,\bar{\lambda})$, and from (\ref{4.3a}), we get
           	 \begin{eqnarray*}
           		h(\alpha s+s^{*}+\lambda\Psi(\alpha u+u^{*},\alpha s+s^{*}))<h_{0},
           		\end{eqnarray*}
           	which contradicts the fact that $s$ is a strict global minimum point of (\ref{4.1}). 
           \end{proof}
       Analogous result to Theorem \ref{27T} is as follows:
           \begin{theorem}   
          	Let $h:R^{n}\rightarrow R$ be a SSQSEP function w.r.t. $\Psi$ on $R^{n}$, $h_{j}:R^{n}\rightarrow R,~ 1\leq j\leq k$, SQSEP functions w.r.t. $\Psi$ on
          	$R^{n}$ and suppose $\emptyset\neq X$ represents set of feasible solutions defined by \ref{4.2}. If $s$ is a fixed local minimum point of (\ref{4.1}), then $s$ is a strict global minimum point of (\ref{4.1}).
          \end{theorem}
      
        \section{\bf{Conclusion}}\label{5conl}
       
      In this article, we have presented the concept of semi strongly $E$-preinvex (semi quasi strongly $E$-preinvex), semi strongly $E$-invex, semi quasi strongly $E$-invex and semi pseudo strongly $E$-invex functions. Sufficient examples have been presented to in support these definitions. The vital relationship and characterizations of these functions have been discussed and several interesting properties of these functions have been established. An application to non-linear programming problem for semi strongly $E$-preinvex (semi quasi strongly $E$-preinvex) functions has also been presented. Our results generalize the previously known results given by different authors; see \cite{Fulga,Hussain2,Hussain3,Jaiswal,Youness1,youness2,Youness3}. This work can be easily explored over Riemannian manifolds in future.
      \vspace{.2cm}
      
      All the concepts mentioned above in this paper are fundamental parts of pure and applied mathematics, which play a vital role in determining the solution of practical problems in real life, such as mathematical programming, optimization problems, variational inequalities, and equilibrium problems; for more details, see \cite{Hussain2,Jaiswal,Jeyakumar,Noor,Yang}.


\begin{thebibliography}{27}
      	
      	\bibitem{Tair} Abou-Tair I. A., Sulaiman W. T., Inequalities via convex functions,
      	International Journal of Mathematical Science. 1999; 22: 543-546.
      	\bibitem{Ben} Ben-Israel A., Mond B., What is invexity? Journal of Australian Mathematical Society. 1986; 28: 1–9.
      	
      	 \bibitem{Bazaara} Bazaraa M.S., Sherali H.D. ,  Shetty C.M., Non-linear Programming: Theory and Algorithms, 2nd ed., Wiley, New York, 1993.
      	
      	\bibitem{Fulga} Fulga C., Preda V., $E$-invex sets and $E$-preinvex functions. European Journal of operational Research. 2009; 192: 737-743.
      	
      	
      	\bibitem{hanson} Hanson  M.A., On sufficiency of Kuhn–Tucker conditions. Journal of Mathematical Analysis and Applications. 1981; 80: 545–550.
       	
     \bibitem{Hussain1} Hussain A., Iqbal A., Quasi strongly $E$-convex functions with applications. Nonlinear Functional Analysis and Applications. 2021; 5: 1077-1089.
     
     
     \bibitem{Akhlad12} Iqbal A., Ali S., Ahmad I., On geodesic $E$-convex sets, geodesic $E$-convex functions and
     E-epigraphs, Journal of Optimization Theory  and Applications. 2012; 55: 239-251.
     
     
     \bibitem{Akhlad11} Iqbal A., Ali S., Ahmad I., Some properties of geodesic semi $E$-convex functions Non-linear Analysis. Nonlinear Analysis. 2011; 74: 6805-6813.
     
     \bibitem{Hussain2} Iqbal A., Hussain A., Non-linear programming problem for strongly $E$-invex set and strongly $E$-preinvex function. RAIRO-Operation Research. 2022; 1397-1410.
     
     \bibitem{Hussain3} Iqbal A., Hussain A., Quasi Strongly $E$-preinvexity and its Relationships with Nonlinear Programming (submitted).
     
      \bibitem{jabar} Jabarootian T., Mahyarinia M. R., Semi E-pseudoinvex and Semi E-quasiinvex
     Functions and Applications. Journal of Mathematical Extension. 2011; 5: 1-12.
     
     \bibitem{Jeyakumar} Jeyakumar W., Strong and weak Invexity in Mathematical Programming. European Journal of operational Research. 1985; 55: 109-125.
     
     \bibitem{Jaiswal} Jaiswal S., Panda  G., "Duality results using higher order generalized $E$-Invex functions". International Journal of Computing Science and Mathematics. 2010; 3: 288-297.
     
      \bibitem{Adem1} Kılıçman A., Saleh  W., On geodesic strongly $E$-convex sets and geodesic strongly $E$-convex functions. Journal of  Inequalities and Applications.  2015: 297.
     
     
     \bibitem{Adem2} Kılıçman A., Saleh  W., On properties of geodesic semilocal $E$-preinvex functions. Journal of Inequality and.  Applications. 2018: 353.
     
     \bibitem{Babli} Kumari B., Jayswal A.,  Some properties of geodesic $E$-preinvex function and geodesic semi $E$-preinvex function on Riemannian manifolds. Operation Research Society of India.  2018; 55: 807–822.
     
     
     \bibitem{Neogy} Mohan  R.S., Neogy K.S., On invex sets and preinvex functions. Journal of Mathematical Analysis and Applications. 1995; 189: 901–908.
     
     \bibitem{Noor} M.A. Noor, Invex equilibrium problems, Journal of Mathematical Analysis and  Application. 2005; 302: 463-475.
     
     \bibitem{Pini} Pini R., Invexity and Generalized Connexity, Optimization: A Journal of Mathematical  programming and Operation Research 1991; 22:  513–
     525.
    
     \bibitem{Suneja} Suneja S.K., Lalitha C.S., Misha, Govil G., Generalized $E$-convex functions in non-linear programming. Indian Journal of Mathematics. 2003; 45: 223–
     240.
      
     \bibitem{Mond} Weir T., Mond  B., Pre-invex functions in multiple objective optimization. Journal of Mathematical Analysis and Applications 1988; 136: 29–38.
      
      \bibitem{Ru Yu} Yu-Ru Syau, Stanley Lee E.,  Semi-E-Preinvex Functions. International Journal of Artificial Life Research. 2010; 31: 31-39.
     
     \bibitem{Yang} Yang X. M., On $E$-convex sets, $E$-convex functions and $E$-convex programming. Journal of Optimization Theory and Applications  2001; 109: 699–704.
     
     \bibitem{Youness1} Youness E. A., $E$-convex sets, $E$-convex functions and $E$-convex
     programming. Journal of Optimization Theory and Applications. 1999; 102: 439-450.
     
     \bibitem{youness2} Youness E. A., Optimality criteria in $E$-convex programming. Chaos Solitions \& Fractals. 2001; 12: 1737-1745.
     
     \bibitem{Youness3} Youness E. A., Emam T., Strongly $E$-convex sets and strongly $E$-convex functions. Journal of Interdisciplinary Mathematics.
     2004, 107-117.
     \bibitem{Youness4} Youness E. A., Emam T.,   Semi strongly $E$-convex functions. Journal of Mathematical and Statistics. 2005; 1: 51-57.
       \end{thebibliography}
  \end{document}